\documentclass[english, 12pt]{amsart}
\usepackage{amssymb}
\usepackage{amsfonts}
\usepackage{amscd}
\usepackage[T1]{fontenc}
\usepackage{color}
\usepackage{mathrsfs} %for the \mathscr

\newcommand{\cCexp}{\cC^{\mathrm{exp}}}
\newcommand{\cCe}{\cC^{\mathrm{e}}}

\newcommand{\Loc}{\mathrm{Loc}}

\def\VF{\mathrm{VF}}

\newcommand{\RF}{{\rm RF}}

\newcommand{\bigdcup}{\mathop{\dot{\bigcup}}}

\def\ac{{\overline{\rm ac}}}

\def\11{{\mathbf 1}}

\def\CC{{\mathbb C}}

\def\FF{{\mathbb F}}

\def\NN{{\mathbb N}}

\def\RR{{\mathbb R}}

\def\ZZ{{\mathbb Z}}

\def\cC{{\mathscr C}}
\def\cD{{\mathcal D}}

\def\cM{{\mathcal M}}

\def\cO{{\mathcal O}}

\def\llp{\mathopen{(\!(}}

\def\rrp{\mathopen{)\!)}}

\newcommand{\C}{{\mathbb C}}

\newcommand{\Z}{{\mathbb Z}}
\newcommand{\Q}{{\mathbb Q}}

\newcommand{\F}{{\mathbb F}}

\newcommand{\rf}{k}
\newcommand{\ri}{\Omega}

\newcommand\ldpo[1][\ri]{{\mathcal L}_{#1}}

\newcommand{\tr}{\operatorname{Tr}}

\def\llp{\mathopen{(\!(}}

\def\rrp{\mathopen{)\!)}}

\newtheorem{thm}[subsection]{Theorem}
\newtheorem{lem}[subsection]{Lemma}
\newtheorem{cor}[subsection]{Corollary}
\newtheorem{prop}[subsection]{Proposition}

\theoremstyle{definition}
\newtheorem{defn}[subsection]{Definition}

\newtheorem{def-prop}[subsection]{Proposition-Definition}
\newtheorem{def-theorem}[subsection]{Theorem-Definition}
\newtheorem{def-lem}[subsection]{Lemma-Definition}

\theoremstyle{remark}

\newtheorem{rem}[subsection]{Remark}

\theoremstyle{plain}

\numberwithin{equation}{subsection}

  {\par\medskip\noindent #1\par\begingroup%
    \advance\leftskip by 1em\advance\rightskip by 1em}%
  {\par\endgroup}

\newcommand{\ord}{\operatorname{ord}}

\begin{document}

\setcounter{tocdepth}{1} % Show subsection in table of contents

\author{Raf Cluckers}
\address{Universit\'e Lille 1, Laboratoire Painlev\'e, CNRS - UMR 8524, Cit\'e Scientifique, 59655
Villeneuve d'Ascq Cedex, France, and,
KU Leuven, Department of Mathematics,
Celestijnenlaan 200B, B-3001 Leu\-ven, Bel\-gium}
\email{Raf.Cluckers@math.univ-lille1.fr}
\urladdr{http://math.univ-lille1.fr/$\sim$cluckers}

\author{Julia Gordon}
\address{Department of Mathematics, University of British Columbia,
Vancouver BC V6T 1Z2 Canada}
\email{gor@math.ubc.ca}
\urladdr{http://www.math.ubc.ca/$\sim$gor}

\author{Immanuel Halupczok}
\address{School of Mathematics, University of Leeds, Leeds, LS2 9JT -- UK}
\email{math@karimmi.de}
\urladdr{http://www.immi.karimmi.de/en.math.html}

\thanks{The author R.C. was supported by the European Research Council under the European Community's Seventh Framework Programme (FP7/2007-2013) with ERC Grant Agreement nr. 615722
MOTMELSUM, by the Labex CEMPI  (ANR-11-LABX-0007-01), and  would like to thank both the Forschungsinstitut f\"ur Mathematik (FIM) at ETH Z\"urich and the IH\'ES for the hospitality during part of the writing of this paper. J.G. was supported by NSERC, and I.H. was supported by the SFB~878 of the Deutsche Forschungsgemeinschaft.}



\title[Transfer principles for Bounds]{Transfer principles for Bounds of motivic exponential functions}

\begin{abstract}
We study transfer principles for upper bounds of motivic exponential functions and for linear combinations of such functions, directly generalizing the transfer principles from \cite{CLexp} and
\cite[Appendix B]{ShinTemp}. These functions come from rather general oscillatory integrals on local fields, and can be used to describe e.g. {Fourier transforms of} orbital integrals. One of our techniques consists in reducing to
simpler functions where the oscillation only comes from the residue field.
\end{abstract}

\maketitle

\section{Introduction}

After recalling concrete motivic exponential functions
and their stability under taking integral transformations, we study transfer principles for bounds of motivic exponential functions and {their} linear combinations.
In this context, \emph{transfer} means switching between local fields with isomorphic residue field (in particular between
positive and mixed characteristic).
By the word \emph{concrete} (in the first sentence), we mean that we work
uniformly in all local fields of large enough residue field characteristic, as opposed to genuinely motivic as done in \cite{CLexp}; this setting is perfectly suited for transfer principles, which are, {indeed,} about local fields.

Our results relate to previously known transfer principles (from \cite{CLexp}, \cite{CGH}, and \cite[Appendix B]{ShinTemp}) as follows.
The principle given by Theorem \ref{thm:transfer-gen} below,
which allows to transfer bounds on motivic exponential functions,
generalizes both the transfer principle of \cite[Proposition 9.2.1]{CLexp},
where, one can say, the upper bound was identically zero, and the transfer principle of \cite[Theorem B.7]{ShinTemp}, where the case without oscillation is treated. {A generalization to $\cCexp$ (instead of $\cC$) of Theorem B.6}
{of} \cite{ShinTemp}({which contains a statement about uniformity across all completions of a given number field rather than a transfer principle)}  is left to future work, since it requires {different, and deeper,}
proof techniques.

The results in this paper are independent of the transfer principles of \cite{CGH} about {e.g., loci of integrability,}
and in fact, our proofs are closer to the ones of \cite{CLexp}, and can avoid the
heavier machinery from \cite{CGH}.

After Theorem \ref{thm:transfer-gen}, we give some further generalizations which treat $\CC$-linear combinations of motivic exponential functions, uniformly in the complex scalars. Specifically, we obtain transfer principles for linear (in-)dependence and for upper bounds of linear combinations of motivic exponential functions (or rather, their specializations for any local field $F$ with large residue field characteristic), see Theorem \ref{thm:transfer-str},  Proposition  \ref{cor:transfer:indep:basic} and Corollary \ref{cor:transfer:indep}.

A key proof technique that we share with \cite{CLexp} consists in reducing from general motivic exponential functions to
simpler functions where the oscillation only comes from additive characters on the residue field. We recall these classes of functions {with their respective oscillatory behavior} in Section \ref{sec:motfun}.

Let us finally mention that the transfer principles of \cite{CLexp} have been applied in \cite{CHL} and \cite[Appendix]{YGordon} to obtain the Fundamental Lemma of the Langlands program in characteristic zero (see also \cite{Nadler}), and the ones of \cite{CGH} have been used in \cite{CGH2} to show local integrability of Harish-Chandra characters in large positive characteristic.
The results of this paper may apply to a wide class of $p$-adic integrals, e.g. orbital integrals and their Fourier transforms.
We will leave the study of such applications to future work.

\subsection*{Acknowledgments}
\hspace{0.5cm}
The authors
{are grateful to} T.~Hales, F.~Loeser and J.~Denef whose influence on the subject of motivic integration is and has always been very important.
Special thanks to T.~Hales for suggesting to look at a statement like {Proposition}~\ref{cor:transfer:indep:basic}.

\section{Motivic exponential functions}\label{sec:motfun}\label{subsub:functions}

In a nutshell, motivic functions are a natural class of functions from {(subsets of)} valued fields to $\CC$,
built from functions on the valued fields that are definable in the Denef-Pas language; the class is
closed under integration. Motivic \emph{exponential} functions are a bigger such class, incorporating additive characters of the valued field.
These functions were introduced in \cite{CLexp}, and the strongest form of stability under integration for these functions was proved in \cite{CGH}. (Constructible functions without oscillation and on a fixed $p$-adic field were introduced earlier by Denef in \cite{Denef1}.)
{We start with recalling} three classes of functions, $\cC$, $\cCe$, and
$\cCexp$, {which have, so to speak, increasing}
oscillatory richness, and each one is stable under integration, see Theorem \ref{thm:mot.int.}.

\subsection*{Motivic  functions}%\label{subsub:functions}

We recall some terminology of \cite{CLoes} and \cite{CLexp}, with the same focus as in \cite{CGH3} (namely uniform in the local field, {as opposed to an approach with Grothendieck rings}).

{Fix a ring of integers $\ri$ of a number field, as base ring.}

\begin{defn}\label{AO}
Let $\Loc_\ri$ be the collection of all triples $(F, \iota , \varpi)$, where $F$ is a non-Archimedean local
field which allows at least one ring homomorphism from $\ri$ to $F$,  the map $\iota:\ri\to F$ is such a ring homomorphism, and $\varpi$ is a uniformizer for the valuation ring of $F$.
Here, by a non-Archimedean local field we mean a finite extension of $\Q_p$ or $\F_p\llp t\rrp$ {for any prime $p$}.

Given an integer $M$, let $\Loc_{\ri, M}$ be the collection of $(F,{\iota},\varpi)$ in $\Loc_\ri$ such that the residue field of $F$ has characteristic at least $M$.
\end{defn}

For a non-Archimedean local field $F$, write $\cO_F$ for its valuation ring with maximal ideal $\cM_F$ and residue field $k_F$ with $q_F$ elements.

We will use the Denef-Pas language with coefficients from $\ri[[t]]$ for {our} fixed ring of integers $\ri$. We denote this language by $\ldpo$.

\begin{defn}\label{DP}
The language $\ldpo$ has three sorts, $\VF$ for the valued field, $\RF$ for the residue field, and a sort for the value group
which we simply call $\ZZ$, since we will only consider structures where it is actually equal to $\ZZ$.
On $\VF$, one has the ring language and coefficients from the ring $\ri[[t]]$. On $\RF$, one has the ring language. On $\ZZ$, one has the Presburger language, namely the language of ordered abelian groups together with constant symbols $0$, $1$, and symbols $\equiv_n$ for each $n>0$ for the congruence relation modulo $n$. Finally, one has the symbols $\ord$ for the valuation map from the valued field minus $0$ to $\ZZ$, and $\ac$ for an angular component map from the valued field to the residue field.
\end{defn}

{It was an important insight of Denef that one has elimination of valued field quantifiers for first order formulas in this language $\ldpo$, and this was worked out by his student Pas in \cite{Pas}. Indeed, quantifier elimination is a first step to understanding the geometry of the definable sets and functions. Another geometrical key result and insight by Denef (\cite{Denef2}, \cite{Pas}, \cite{CLoes}) is the so-called cell decomposition, which is behind Proposition \ref{lem:presburger-fam}.}

The language $\ldpo$ is interpreted in any $(F,\iota,\varpi)$ in $\Loc_\ri$ in the obvious way, {where $t$ is interpreted as $\varpi$ and} where $\ac$ is defined by
$$
\ac(u\varpi^\ell) = \bar u \mbox{ and } \ac(0)=0
$$
for any $u\in \cO_F^\times$ and $\ell\in\ZZ$, $\bar u$ being reduction modulo $\cM_F$.
We will abuse notation by notationally identifying $F$ and $(F,\iota,\varpi)\in \Loc_\ri$.

Any $\ldpo$-formula $\varphi$ gives a subset $\varphi(F)$  of $F^n\times k_F^m\times \ZZ^r$ for $F\in \Loc_\ri$ for some $n,m,r$ only depending on $\varphi$, by taking the  $F$-rational points on $\varphi$ in the sense of model theory (see Section 2.1 of \cite{CGH3} for more explanation). This leads us to the following handy definition.

\begin{defn}\label{defset}
A collection
$X = (X_F)_{F \in \Loc_{\ri,M}}$ of subsets $X_F\subset F^n\times \rf_F^m\times \Z^r$ for some $M,n,m,r$ is called a \emph{definable set} if there is an
$\ldpo$-formula $\varphi$ such that $X_F = \varphi(F)$ for each
$F$ in $\Loc_{\ri,M}$ (see Remark \ref{rem:largep}).
\end{defn}

By Definition \ref{defset}, a ``definable set'' is actually a collection of sets indexed by $F\in \Loc_{\ri,M}$; such practice is often used in model theory and also in algebraic geometry.
A particularly simple definable set is $(F^n\times k_F^m\times \ZZ^r)_F$, for which we use
the simplified notation
$\VF^n\times \RF^m\times \ZZ^r$.
We apply the typical set-theoretical notation to definable sets $X, Y$, e.g.,
$X \subset Y$ (if $X_F \subset Y_F$ for each $F \in \Loc_{\ri,M}$ for some $M$), $X \times Y$, and so on, {which may increase $M$ if necessary}.

\begin{defn}\label{deffunct}
For definable sets $X$ and $Y$, a collection $f = (f_F)_F$ of functions $f_F:X_F\to Y_F$ for $F\in\Loc_{\ri,M}$ for some $M$ is called a definable function and denoted by $f:X\to Y$ if
the collection of graphs of the $f_F$ is a definable set.
\end{defn}

\begin{rem}\label{rem:largep}
For a definable set $X$ as in Definition \ref{defset}, we are usually only interested in $(X_F)_{F \in \Loc_{\ri,M}}$ for $M$ sufficiently big, and thus, we often allow ourselves to replace $M$ by a larger number if necessary, {without saying so explicitly; also} the uniform objects defined below in Definitions \ref{motfun}, \ref{expfun}, and so on, are only interesting for $M$ sufficiently large.
In model theoretic terms, we are {using} the theory of {all} non-archimedean local fields, together {with, for each $M>0$,  an axiom stating that the residue characteristic is at least $M$.}
Note however that a more general theory of uniform integration which works uniformly in all local fields of mixed characteristic (but not in local fields of small positive characteristic), is under development in \cite{CHallp} {and will generalize \cite{CLbounded}.}
\end{rem}

For motivic  functions, definable functions are the building blocks, as follows.

\begin{defn}\label{motfun}
Let $X = (X_F)_{{F\in \Loc_{\ri,M}}}$ be a definable set.
A collection $H = (H_F)_F$ of functions $H_F:X_F\to\RR$ is called \emph{a motivic function} on $X$ if
there exist integers
$N$, $N'$, and $N''$, nonzero integers $a_{i\ell}$, definable functions $\alpha_{i}:X\to \ZZ$ and $\beta_{ij}:X\to \ZZ$,
and definable sets  $Y_i\subset X\times \RF^{r_i}$ such that for all $F\in \Loc_{\ri, {M}}$ and all $x\in X_F$
$$
H_F(x)=\sum_{i=1}^N    \# Y_{i,F,x} \cdot  q_F^{\alpha_{iF}(x)} \cdot \big( \prod_{j=1}^{N'} \beta_{ijF}(x) \big) \cdot \big( \prod_{\ell=1}^{N''} \frac{1}{1-q_F^{a_{i\ell}}} \big),
$$
where $Y_{i,F,x}$ is the finite set $\{y\in k_F^{r_i}\mid (x,y)\in Y_{i,F}\}$.

We write $\cC(X)$ to denote the ring of motivic functions on $X$.
\end{defn}

The precise form of this definition is motivated by
the property that motivic functions behave well under integration (see Theorem \ref{thm:mot.int.}).

\subsection*{Motivic exponential functions}\label{subsub:expfunctions}

For any local field $F$, let $\cD_F$ be the set of the additive characters $\psi$ on $F$ that are trivial on the maximal ideal $\cM_F$ of $\cO_F$,
nontrivial on $\cO_F$, and such that, for $x\in \cO_F$, one has
\begin{equation}\label{psiu}
\psi(x) =  {\mathbf{e}}(\tr_{k_F/\FF_{p}} (\bar x))
\end{equation}
with $\bar x$ the reduction of $x$ modulo $\cM_F$ and where $q_F$ is an integer power of the prime number $p$, and where $\mathbf{e}:\FF_p\to \CC$ sends $a\in \{0,\ldots,p-1\}$ to
$\exp(\frac{ 2\pi i a}{p})$ for some fixed complex square root $i$ of $-1$.
Expressions involving additive characters of $p$-adic fields often give rise to exponential sums, and this explains the term ``exponential'' in the definition below.

\begin{defn}\label{expfun}
Let $X  = (X_F)_{ {F\in \Loc_{\ri,M}}}$ be a definable set.
A collection $H = (H_{F,\psi})_{F,\psi}$ of functions $H_{F,\psi}:X_F\to\CC$ for $F\in \Loc_{\ri,M}$ and $\psi\in \cD_F$ is called \emph{a motivic exponential function} on $X$ if
there exist integers $N>0$ and $r_i\geq 0$, motivic functions $H_i=(H_{iF})_F$ on $X$, definable sets $Y_i\subset X\times {\RF}^{r_i}$ and definable functions $g_i:Y_i\to \VF$ and $e_i:{Y_i}\to \RF$ for $i=1,\ldots,N$, such that for all $F\in \Loc_{\ri,M}$, all $\psi\in \cD_F$ and all $x\in X_F$
\begin{equation}\label{fexp}
H_{F,\psi}(x)=\sum_{i=1}^N   H_{iF}(x)\Big( \sum_{y \in Y_{i,F,x}}\psi\big( g_{iF}(x,y)  { + e_{iF}(x,y) }     \big)\Big),
\end{equation}
{where $\psi(a+v)$ for $a\in F $ and $v\in k_F$, by abuse of notation,  is defined as $\psi ( a + u)$, with $u$ any unit in $\cO_F$ such that $\bar u=v$, which is well defined by (\ref{psiu}).}
We write $\cCexp(X)$ to denote the ring of motivic exponential functions on $X$.
Define the subring $\cCe(X)$ of $\cCexp(X)$ consisting of those functions $H$ as in (\ref{fexp}) such that all $ {g_{iF}}$ are identically vanishing. Note that for $H \in \cCe(X)$, $H_{F,\psi}$ does not depend on  {$\psi\in \cD_F$ because of (\ref{psiu})}, so we will just write $H_F$ instead.
\end{defn}

Compared to Definition \ref{motfun}, the counting operation $\#$ has been replaced by taking exponential sums, which makes the motivic exponential functions a richer class than the motivic functions.
Indeed, note that the sum as above gives just $\#(Y_{iF})_x$ in the case that  {$g_{iF}=0$ and $e_{iF}=0$}.

\subsection*{Integration}

To integrate a motivic function $f$ on a definable set $X$, we need a uniformly given family of measures on each $X_F$. For $X = \VF$, we put the Haar measure on $X_F = F$ so that $\cO_F$ has measure $1$; on
$k_F$ and on $\ZZ$, we use the counting measure and for $X\subset \VF^n\times \RF^m\times \ZZ^r$ we use the  {measure on $X_F$ induced by the product measure on $F^n\times k_F^m\times \ZZ^r$}. To obtain other  {motivic} measures on
definable sets $X$, one can also use measures associated to ``definable volume forms'', see Section 2.5 of \cite{CGH3},  {\cite[\S 8]{CLoes}}, and Section 12 of \cite{CLbounded}.

Maybe the most important aspect of these motivic functions is that they have nice and natural properties related to integration, see e.g.~the following theorem about stability, which  {generalizes} Theorem 9.1.4 of
\cite{CLexp} (see also Theorem 4.1.1 of \cite{CLexp}).

\begin{thm}[{\cite[Theorem 4.3.1]{CGH}}]\label{thm:mot.int.}
Let $f$ be in $\cC(X\times Y)$, resp. in $\cCe(X\times Y)$ or  {in} $\cCexp(X\times Y)$, for some definable sets $X$ and $Y$, with $Y$ equipped with a motivic measure  {$\mu_Y$}. Then there exist a  function  ${I}$ in $\cC(X)$, resp.~in $\cCe(X)$ or $\cCexp(X)$ and an integer $M>0$ such that for each $F\in \Loc_{\ri,M}$, each $\psi\in \cD_F$
and for each $x\in X_F$ one has
$$
{I}_F(x) = \int_{y\in Y_F}f_F(x,y)\,d\mu_{Y_F}, \mbox{ resp. } I_{F,\psi}(x) = \int_{y\in Y_F}f_{F,\psi}(x,y)\,d\mu_{Y_F},
$$
whenever the function $Y_F\to\CC:y\mapsto f_F(x,y)$, resp.~$y\mapsto f_{F,\psi}(x,y)$, is in $L^1$.
\end{thm}
\begin{proof}
The cases $\cC$ and $\cCexp$ are treated in \cite{CGH}, Theorems 4.3.1 and 4.4.3.
The proof for $\cCexp$ in \cite{CGH} goes through also for $\cCe$. (A more direct and simpler proof for $\cCe$ can also be given, by reducing to the case for $\cC$ using residual parameterizations as in Definition 4.5.1 of \cite{CGH}.)
\end{proof}

\section{Transfer principles for bounds and linear combinations} % for motivic exponential functions}

In this section, we state the main results of this article.

The following statement allows one to transfer bounds which are known {for local fields of}
characteristic zero to local fields of positive characteristic, and vice versa.

\begin{thm}[Transfer principle for bounds]\label{thm:transfer-gen}
Let $X$ be a definable set, let  $H$ be in  $\cCexp(X)$ and let $G$ be in $\cCe(X)$.
Then there exist $M$ and $N$ such that, for any $F\in \Loc_{\ri, M}$,
the following holds.
If
\begin{equation}\label{transfer-gen}
|H_{F,\psi} (x)|_\C \le |G_{F}(x)|_\CC   \mbox{ for all } (\psi,x) \in \cD_F\times X_F
\end{equation}
then, for any local field $F'$ with the same residue field as $F$, one has
\begin{equation}\label{transfer-gen:eq}
|H_{{F'},\psi} (x)|_\C \le N \cdot |G_{F'}(x)|_\CC \mbox{ for all } (\psi,x) \in \cD_{F'}\times X_{F'}.
\end{equation}
Moreover, one can take $N=1$ if $H$ lies in $\cCe(X)$.
\end{thm}

As mentioned in the introduction, the case where $G=0$ is \cite[Proposition 9.2.1]{CLexp}, and the case that both $H$ and $G$ lie in $\cC(X)$ is {\cite[Theorem B.7]{ShinTemp}}.

We also show the following strengthening of Theorem \ref{thm:transfer-gen}, for linear combinations.

\begin{thm}[Transfer principle for bounds of linear combinations]\label{thm:transfer-str}
Let $X$ be a definable set, let  $H_i$ be in  $\cCexp(X)$ for $i=1\,\ldots,\ell$, and let $G$ be in $\cCe(X)$.
Then there exist $M$ and $N$ such that, for any $F\in \Loc_{\ri, M}$,
the following holds for any $c=(c_i)_i$ in $\CC^\ell$.
If
\begin{equation}\label{str:transfer-gen}
|\sum_{i=1}^\ell c_i H_{i,F,\psi} (x)|_\C \le |G_{F}(x)|_\CC   \mbox{ for all } (\psi,x) \in \cD_F\times X_F
\end{equation}
then, for any local field $F'$ with the same residue field as $F$, one has
\begin{equation}\label{str:transfer-gen:eq}
|\sum_{i=1}^\ell c_i H_{i,F',\psi} (x)|_\C \le N \cdot |G_{F'}(x)|_\CC   \mbox{ for all } (\psi,x) \in \cD_{F'}\times X_{F'}.
\end{equation}
Moreover, one can take $N=1$ if  {the $H_i$ lie} in $\cCe(X)$.
\end{thm}

The key improvement of Theorem \ref{thm:transfer-str} (compared to Theorem~\ref{thm:transfer-gen}) is that the choice of $M$ and $N$ works uniformly in $c$.

Although in our proofs, the integer $N$ of Theorems \ref{thm:transfer-gen} and \ref{thm:transfer-str} appears naturally, it is not unconceivable that one can take $N$ close to $1$ even when $H$ does not lie in $\cCe(X)$.

The first statement of the following proposition allows to transfer linear (in-)dependence of motivic exponential functions;
its proof follows quite directly from transfer of  {identical vanishing of motivic functions}
(the $G=0$ case of  {Theorem} \ref{thm:transfer-gen}). The second statement describes how the
coefficients in a linear
{relation} can depend on the local field,  {the additive character, and } parameters; see below for more explanation.

\begin{prop}[Transfer principle for linear dependence]\label{cor:transfer:indep:basic}

Let $X$ and $Y$ be definable sets and let $H_i$ be in  $\cCexp(X \times Y)$ for $i=1\,\ldots,\ell$.
\begin{enumerate}
 \item
There exists $M$ such that, for any $F,F'\in \Loc_{\ri, M}$ with $k_F\cong k_{F'}$,
the following holds:

If for each $\psi\in\cD_F$ and each $y \in Y_F$ the functions $H_{i,F,\psi}(\cdot, y):X_F\to\CC$ for $i=1,\ldots,\ell$ are linearly dependent,
then, also for each $\psi \in \cD_{F'}$ and each $y \in Y_{F'}$, the functions $H_{i,F',\psi}(\cdot, y)$ are linearly dependent.
\item
Let moreover $G$ be in $\cCexp(X \times Y)$. Then there exists a definable set $W$ and
functions $C_i$ and $D$ in $\cCexp(W \times Y)$ such that the following holds for $M$ sufficiently big.
For every $F \in \Loc_{\ri, M}$, for every $\psi\in \cD_F$, and for every $y \in Y_F$, if
the functions $H_{i,F,\psi}(\cdot,y)$ (on $X_F$) are linearly independent and
\[
G_{F,\psi}(\cdot, y) = \sum_{i=1}^{\ell} c_i H_{i, F, \psi}(\cdot, y)
\]
for some $c_i \in \CC$, then $D_{F, \psi}(\cdot, y)$ is not identically zero on $W_F$, and for all $w\in W_F$
\[
D_{F, \psi}(w,y) c_i=C_{i, F, \psi}(w,y).
\]
\end{enumerate}
\end{prop}

The second part of the proposition, essentially, states that the coefficients $c_i$ are ratios of motivic exponential functions.
{However,} our proof needs an additional parameter $w$ to  {write the $c_i$ as ratios}: both, $C_i$ and $D$ depend on $w$ and only their quotient is independent of $w$,  {for $w$ with $D_{F,\psi}(w,y)$ nonzero}.
Note that despite this complication, the proposition permits to apply transfer principles to the constants $c_i$.
(One  {of course does not need $w$ if there is} a definable function $h\colon Y \to W$ such that  {$D_{F,\psi}\circ h_F$ is}  {nowhere zero}.)

Proposition \ref{cor:transfer:indep:basic} naturally applies also in case that the $H_i$ and $G$ are in $\cCexp(Z)$ for some definable subset $Z$ of $X\times Y$, instead of in $\cCexp(X \times Y)$. Indeed, one can extend the $H_i$ by zero outside $Z$ and apply the proposition to these extensions.

Finally, we note the following corollary of Theorem~\ref{thm:transfer-str},
showing that the complex coefficients of a linear relation between motivic exponential functions stay \emph{the same}
(regardless of their motivic interpretation as in Proposition~\ref{cor:transfer:indep:basic} above) in situations where these coefficients are independent of the additive character. This independence is a strong assumption, but note that it in particular
applies to arbitrary linear relations of motivic non-exponential functions.

\begin{cor}[Transfer principle for coefficients of linear relations]\label{cor:transfer:indep}
Let $X$ be a definable set and let  $H_i$ be in  $\cCexp(X)$ for $i=1\,\ldots,\ell$.
Then there exists $M$ such that, for any $F\in \Loc_{\ri, M}$,
the following holds for any $c=(c_i)_i$ in $\CC^\ell$.

If
$$
\sum_{i=1}^\ell c_i H_{i,F,\psi}  =0 \mbox{ on $X_F$ for all $\psi\in \cD_F$,}
$$
then, for any $F'\in \Loc_{\ri, M}$ with $k_F \cong k_{F'}$, one also has
$$
\sum_{i=1}^\ell c_i H_{i,F',\psi'}  =0 \mbox{ on $X_{F'}$ for all $\psi'\in \cD_{F'}$.}
$$
\end{cor}

\begin{proof}
Just apply Theorem \ref{thm:transfer-str} with $G=0$.
\end{proof}

\section{Proofs of the transfer principles}\label{theproofs}

Before proving Theorem \ref{thm:transfer-gen}, we give a proposition relating the square of the complex modulus of a motivic exponential function to the complex modulus of a function where the oscillation only comes from the residue field.

\begin{prop}\label{from.exp.to.e}
Let $H$ be in $\cCexp(X)$ for some definable set $X$. Then there exist
$ {\tilde H}$ in $\cCe(X)$ and integers $M$ and $N$  such that for all $F\in \Loc_{\ri, M}$ the following hold for all $x\in X_F$.
\begin{enumerate}
 \item
 There is $\psi_1$ in $\cD_F$ (depending on $x$) such that
$$
\frac{1}{N}| {\tilde H}_F(x)|_\CC \leq | H_{F,\psi_1}(x)|_\CC^2.
$$
\item
For all $\psi$ in $\cD_F$, one has
$$
|H_{F,\psi}(x)|_\CC^2 \leq | {\tilde H}_F(x)|_\CC.
$$
\end{enumerate}
\end{prop}

The proof of Proposition \ref{from.exp.to.e} uses an elementary result from Fourier analysis, which we now recall.
\begin{lem}\label{four}
Consider a finite abelian group $G$ with dual group $\hat G$ and with $|G|$ elements.
For any function $f:G\to\CC$ one has
$$
\frac{1}{|G|} \|\hat f\|_{\sup} \leq \| f\|_{\sup} \leq \|\hat  f\|_{\sup}
$$
where $\|\cdot\|_{\sup} $ is the supremum norm and $\hat f$ the Fourier transform of $f$, namely
$$
\hat f ( \varphi ) = \sum_{x\in G} f(x) \varphi(x) \mbox{ for } \varphi \in \hat G.
$$
\end{lem}
\begin{proof}
Clearly one has
$$
\| f \|_{\sup} \leq \| f\|_2 \leq \sqrt{|G|}\| f\|_{\sup},
$$
and similarly for $\hat f$, where $\| \cdot\|_2$ is the $L_2$-norm, namely $\| f\|_2 = \sqrt{\sum_{g\in G}|f(g)|_\CC^2}$.
By Plancherel  {identity} one has
$$
\sqrt{|G|}\| f\|_2  = \| \hat f\|_2.
$$
The lemma follows.
\end{proof}

\begin{cor}\label{linear}
Consider a finite abelian group $G$ with dual group $\hat G$.
Consider a function
$$f:\hat G\to\CC:\varphi\mapsto \sum_{j=1}^s  c_j \varphi(y_j)$$
for some complex numbers $c_j$ and some distinct $y_j\in G$. Then there exists $\varphi_0\in \hat G$ with
$$
 \sup_{1\le j \le s} |c_j|_{\CC} \leq |f(\varphi_0)|_{\CC}.
$$
\end{cor}

Note that Corollary \ref{linear} generalizes Lemma 9.2.3 of \cite{CLexp},  {a} basic ingredient for proving the transfer principle  \cite[Proposition 9.2.1]{CLexp}.

We will use the  {simple} fact that for $n$ complex numbers $a_i$, one has
\begin{equation}\label{basicf}
\sum_{i=1}^n |a_i|_\CC^2 \leq (\sum_{i=1}^n |a_i|_\CC)^2 \leq n \cdot \sum_{i=1}^n |a_i|_\CC^2.
\end{equation}

\begin{proof}[Proof of Proposition \ref{from.exp.to.e}]
 {Recall} that we allow ourselves to increase $M$ whenever necessary without further mentioning.  {By}  ``for every $F$'' we shall always mean {for}  $F \in \Loc_{\ri, M}$.

Consider a general $H$ in $\cCexp(X)$ and write it as in (\ref{fexp}):
\begin{equation}\label{recall:fexp}
H_{F,\psi}(x)=\sum_{i}   H_{iF}(x)\Big( \sum_{y \in Y_{i,F,x}}\psi\big(g_{iF}(x,y)   +   e_{iF}(x,y)\big)\Big).
 \end{equation}
We will start by grouping the summands of the sum over $y$ according to the value of $g_{iF}(x,y)$ modulo $\cO_F$.
This is done as follows. For each $x \in  {X_F}$, the union of the images $A_{F,x} := \bigcup_i g_{iF}(Y_{i,F,x})$ is finite.
Therefore, the cardinality $\#A_{F,x}$ is bounded by some $N' > 0$ (independently of $x$ and $F$), and by cell decomposition
(in the form of Theorem 7.2.1 of \cite{CLoes}),
there exists a definable set $X'\subset X\times \RF^t$ (for some $t\geq 0$) and a definable function $g':X'\to \VF$
inducing a bijection $X'_{F,x} \to A_{F,x}$ for every $F$ and $x$ (where $X'_{F,x}$ is the fiber of $X'_F$ over $x \in X_F$).
This allows us to write $H$ as
\begin{equation}\label{H1}
H_{F,\psi}(x) = \sum_{x' \in X'_{F,x}}\psi(g'_F(x')) H'_F(x').
\end{equation}
for a suitable $H' \in \cCe(X')$; indeed, we can take $H'$ such that
\[
H'_F(x') = \sum_{i} H_{iF}(\pi(x'))
\sum_{\substack{y \in Y_{i,F,x}\\ g_{iF}( {x},y) = g'_F(x')\\   {\pi(x')=x } }} {\psi}\big(e_{iF}( {x},y)\big)
,
\]
where $\pi\colon X' \to X$ is the projection  {and with notation as in (\ref{fexp}) concerning $\psi(\xi)$ for $\xi \in k_F$, which does not depend on $\psi$ since it is fixed by (\ref{psiu}).}

This construction ensures that for $x', x'' \in X'_{F,x}$ with $x' \ne x''$, we have
$g'_{F}(x') \ne g'_{F}(x'')$. We can even achieve that for such $x', x''$ we have
\begin{equation}\label{<0}
\ord(g'_F(x') - g'_F(x''))<0,
\end{equation}
by modifying $g_{iF}$ and $e_{iF}$ in (\ref{recall:fexp}) in such a way that
$g'_{F}(x') \ne g'_{F}(x'')$ already implies (\ref{<0}). To this end,
replace $g_{iF}(x,y)$ by the arithmetic mean of the (finite) set
$A_{F,x} \cap (g_{iF}(x,y) + \cO_F)$ and change $e_{iF}$ {, using the additivity of $\psi$,} to make up for this modification.

Let $G'$ be a function in $\cCe(X')$ such that for all $F$,
\begin{equation}\label{H2}
G'_F = |H'_{F}|_\CC^2.
\end{equation}
Such $G'$ exists by multiplying (uniformly in $F$) $H'_F$ with its complex conjugate which is constructed by replacing the arguments {(appearing in $H'$)} of the additive character on the residue field by their additive inverses,  {similarly to} the proof of Lemma 4.5.9 of \cite{CGH}. Now define $ {\tilde H}$ such that
\begin{equation}\label{cH}
 {\tilde H}_{F}(x) = N'\cdot \sum_{x',\ \pi_F(x')=x } G'_F(x')
\end{equation}
for each $F$ and each $x\in X_F$,
and let $N$ be $N'{}^{2}$. We claim that $ {\tilde H}$ and $N$ are as desired. Firstly, $ {\tilde H}$ lies in $\cCe(X)$ by Theorem \ref{thm:mot.int.}.
From (\ref{basicf}), (\ref{H1}) and (\ref{H2})
it follows that
$$
|H_{F,\psi}(x)|_\CC^2 \leq | {\tilde H}_F(x)|_\CC \mbox{ for all $(\psi,x)$ in $\cD_F\times X_F$.}
$$
We now show that for each $x\in X_F$ there is $\psi_1$ in $\cD_F$ such that
\begin{equation}\label{H3}
\frac{1}{N}| {\tilde H}_F(x)|_\CC \leq | H_{F,\psi_1}(x)|_\CC^2.
\end{equation}
Fix $F$ and $x\in X_F$. From Corollay \ref{linear}, applied to a large enough finite subgroup $G$ of $F/\cO_F$ so that $G$ contains $g_F'(x')\bmod \cO_F$ for all $x'$ with $\pi(x')=x$,
one finds $\psi_1$ in $\cD_F$ such that
$$
 \sup_{x',\ \pi_F(x')=x} |H'_F(x')|_\CC \leq |H_{F,\psi_1}(x)|_{\CC}.
$$
Hence, from (\ref{basicf}) again,
$$
 \sum_{x',\ \pi_F(x')=x} |H'_F(x')|^2_\CC \leq N'|H_{F,\psi_1}(x)|^2_{\CC},
$$
and thus
$$
\frac{1}{N}  {\tilde H}_F(x) =  \frac{N'}{N}\sum_{x',\ \pi_F(x')=x} |H'_F(x')|^2_\CC \leq \frac{N'{}^2}{N} |H_{F,\psi_1}(x)|^2_{\CC} = |H_{F,\psi_1}(x)|^2_{\CC}.
$$
 {This shows (\ref{H3}).}
\end{proof}

We will also use the following generalization of Proposition B.8 of the appendix B of \cite{ShinTemp}.
Intuitively, it says that functions in $\cCe(S)$ (for arbitrary definable $S$)
only depend on value group and residue field information.

\begin{prop}\label{lem:presburger-fam}
Let $H$ be in $\cCe(S\times B)$ for some definable sets $S$ and $B$. Then there exist a definable function $f : S\times B \to \RF^m \times \ZZ^r\times B$ for some $m\geq 0$ and $r\geq 0$, which makes a commutative diagram with both projections to $B$, and a function $G$ in $\cCe(\RF^m \times \ZZ^r \times B)$ such that, for some $M$ and all $F$ in $\Loc_{\ri,M}$, the function $H_F$ equals the function $G_F\circ f_F$, and such that $G_F$ vanishes outside the range of $f_F$.
\end{prop}

\begin{proof}
The proof is similar to the one for Proposition B.8 in Appendix B of \cite{ShinTemp}.
Let us write $S\subset \VF^n \times \RF^a \times \ZZ^b$ for some integers $n,a$ and $b$. It is enough to prove the lemma when $n=1$ by a finite recursion argument.
The case $n=1$ follows from the Cell Decomposition Theorem 7.2.1 from \cite{CLoes}. Indeed, this result can be used to push the domains of all appearing definable functions in the  {build-up} of $H$ into a set of the form $\RF^m \times \ZZ^r$,
forcing them to have only residue field variables and value group variables.
\end{proof}

\begin{proof}[Proof of Theorem \ref{thm:transfer-gen}]
By Proposition \ref{from.exp.to.e} it is enough to consider the case that $H$ lies in $\cCe(X)$ and to show that one can take $N=1$ in this case.
Suppose that $X$ is a definable subset of
$\VF^n \times \RF^m \times \ZZ^r$. %We give a proof by induction on $m$.
In the case that $n=0$, the proof goes as follows.
By quantifier elimination, any finite set of formulas needed
to describe $H$ and $G$ can be taken to be without valued field quantifiers. It follows that
\begin{equation}H_{F_1}=H_{F_2} \mbox{ and } G_{F_1}=G_{F_2}\end{equation}
for $ {F}_1$ and $F_2$ in $\Loc_{\ri, M}$ with $k_{F_1}\cong k_{F_2}$
and $M$ large enough, and up to identifying $k_{F_1}$ with $k_{F_2}$. This implies the case $n=0$ with $N=1$.

Now assume $n>0$. By
Proposition \ref{lem:presburger-fam}, there is a
definable function
\begin{equation}
f:X \to
\RF^{m'} \times \ZZ^{r'}
\end{equation}
for some $m'$, $r'$, and $\tilde H\in \cCe(\RF^{m'} \times \ZZ^{r'})$ and $\tilde G\in  {\cCe}(\RF^{m'} \times \ZZ^{r'})$, such that $H= {\tilde H \circ f}$ and  {$G=\tilde G\circ f$ and such that $\tilde H$ and $\tilde G$ vanish outside the range of $f$.} We finish the case of $H$ in $\cCe (X)$ by applying the case $n=0$ to $\tilde H$ and $\tilde G$.
\end{proof}

In order to prove Theorem \ref{thm:transfer-str}, we will need the corresponding strengthening of Proposition \ref{from.exp.to.e}, which goes as follows.

\begin{prop}\label{from.exp.to.e:str}
Let $H_i$ be in $\cCexp(X)$ for some definable set $X$ and for $i=1,\ldots,\ell$ for some $\ell>0$. Then there exist integers $M$ and $N$, {and functions} $ {\tilde H}_{i,s}$ in $\cCe(X)$ for $i,s=1,\ldots,\ell$, such that for all $F\in \Loc_{\ri, M}$ the following  {conditions} hold for all $x\in X_F$ and all $c=(c_i)_i$ in $\CC^\ell$.
\begin{enumerate}
\item
There is $\psi_1$ in $\cD_F$ (depending on $x$  {and $c$}) such that
$$
\frac{1}{N} | \sum_{i,s=1}^\ell c_i\bar c_s  {\tilde H}_{i,s,F}(x)|_\CC \leq | \sum_i c_i H_{i,F,\psi_1}(x)|_\CC^2.
$$
\item
For all $\psi$ in $\cD_F$, one has
$$
|\sum_i c_i H_{i,F,\psi}(x)|_\CC^2 \leq  | \sum_{i,s=1}^\ell c_i\bar c_s
 {\tilde H}_{i,s,F}(x)|_\CC.
$$
\end{enumerate}
\end{prop}
\begin{proof}
We start by applying the construction from the beginning of the proof of Proposition \ref{from.exp.to.e} to each of our functions
$H_{i,F,\psi}$, i.e., we write each of them in the form
\begin{equation}\label{str:H1}
H_{i,F,\psi}(x) = \sum_{x' \in X'_{F,x}}\psi(g'_{F}(x')) H'_{i,F}(x'),
\end{equation}
where $X'\subset X\times\RF^t$ has  {finite fibers $X'_{F,x}$ which are bounded uniformly in $x\in X_F$ and in $F$,}
$H'_i$ lies in $\cCe(X')$,  $g':X'\to \VF$ is definable, and such that
\begin{equation}\label{str:<0}
\ord(g'_{F}(x') - g'_{F}(x''))<0
\end{equation}
for any $x', x'' \in X'_{F,x}$ with $x' \ne x''$.

We can do this in such a way that neither $X'$ nor $g'$ depends on $i$. Indeed,
first do the construction for each $H_{i,F,\psi}$ separately, yielding sets $X'_i$ and functions $g'_i$.
Then let $X' := \bigdcup_i X'_i$ be the disjoint union, set $g'_F(x') := g'_{i,F}(x')$ if $x' \in X'_i$ and
extend $H'_{i,F}(x')$ from $X'_i$ to $X'$ by $0$.
Finally, note that the same construction as in the proof of Proposition \ref{from.exp.to.e} allows us to assume that
(\ref{str:<0}) holds on the whole of $X'$.

Let $G'_{i,s}$ be functions in $\cCe(X')$ such that
\begin{equation}\label{str:H2}
\sum_{i,s=1}^\ell c_i\bar c_s G'_{i,s,F}(x')   = | {\sum_{i=1}^\ell}  c_i H'_{i,F}(x')|_\CC^2.
\end{equation}
(for all $F$ and all $x'\in X'_F$).
Such $G'_{i,s}$ exist
%similarly as
 {by a similar argument to the one} explained for $G'$ in the proof of Proposition \ref{from.exp.to.e}.
Now use Theorem \ref{thm:mot.int.}  {for each $i$ to define
$ {\tilde H}_{i,s}$ in $\in \cCe(X)$ satisfying}
\begin{equation}\label{str:cH}
 {\tilde H}_{i,s,F}(x) = N'\cdot \sum_{x' \in X'_{F,x} } G'_{i,s,F}(x'),
\end{equation}
where $N' \in \NN$
is some constant which we will fix later.

We claim that for a suitable choice of $N'$, the  {functions} $ {\tilde H}_{i,s}$ are as desired.
Indeed, we have the following, where  {the relations} ``$\approx_1$'' and ``$\approx_2$''
are explained below.
\begin{align*}
\left|\sum_i c_i H_{i,F,\psi}(x)\right|_\CC^2
&\overset{( {\ref{str:H1}})}{=}
\left|\sum_{x'\in X'_{F,x}}\psi(g'_F(x'))\sum_{i} c_i H'_{i,F}(x')\right|_\CC^2
\\
&\approx_1
\left|\sum_{x'\in X'_{F,x}}\bigg|\sum_{i} c_i H'_{i,F}(x')\bigg|\right|_\CC^2
\\
&\approx_2
\sum_{x'\in X'_{F,x}}\bigg|\sum_{i} c_i H'_{i,F}(x')\bigg|_\CC^2
\\
&\overset{\substack{(\ref{str:H2})\\(\ref{str:cH})}}{=}
\frac1{N'} \bigg| \sum_{i,s=1}^\ell c_i\bar c_s  {\tilde H}_{i,s,I,F}(x)\bigg|_\CC
\end{align*}
The two sides of ``$\approx_2$'' differ at most by a constant, by the  {simple fact (\ref{basicf}) and since the sets $X'_{F,x}$ are finite sets which are bounded uniformly in $x\in X_F$ and $F$.}
At ``$\approx_1$'', we have ``$\le$'', which already implies (2) of the proposition for a suitable choice of $N'$, and
we obtain an estimate in the other direction in the same way as
in the proof of Proposition \ref{from.exp.to.e}: By Corollary~\ref{linear}, and using
(\ref{str:<0}), for each $F$ and each $x$, there exists a $\psi_1 \in \cD_F$ such that
\[
\left|\sum_{x'\in X'_{F,x}}\psi_1(g'_F(x'))\sum_{i} c_i H'_{i,F}(x')\right|_{\CC}
\ge
\sup_{x'\in X'_{F,x}}\bigg|\sum_{i} c_i H'_{i,F}(x')\bigg|
;
\]
now use once more that the cardinality of $X'_{F,x}$ is  {uniformly} bounded to replace the supremum over $x'$ by the sum,  {and to obtain (1) of the Proposition}.
\end{proof}

\begin{proof}[Proof of Theorem \ref{thm:transfer-str}]
By Proposition \ref{from.exp.to.e:str} it is enough to consider the case that the $H_i$ lie in $\cCe(X)$ and to show that one can take $N=1$ in this case. But this case is proved as the proof for the corresponding case of Theorem \ref{thm:transfer-gen}.
\end{proof}

\medskip

It remains to prove Proposition~\ref{cor:transfer:indep:basic}.  {We do this by reducing to the transfer principle of \cite[Proposition 9.2.1]{CLexp}}. The main ingredient  {for this reduction} is the following  {classical result}, which
shows that  a finite collection of
functions being linearly dependent is equivalent to some other function  {that can be constructed from this collection}  {being constantly zero}.

\begin{lem}\label{lem:indep}
Let $f_i$ be complex-valued functions on some set $A$ for $i=1,\ldots,n$. Then there exists nonzero $c=(c_i)_{i=1}^n$ in $\CC^n$ such that {the function}
$
\sum_{i=1}^n c_i f_i
$
is identically vanishing on $A$ if and only if the determinant of  {the matrix}
$$
(f_i(z_j))_{i,j}
$$
is identically vanishing on $A^n$,
where the $z_j$ are distinct variables, running over $A$ for $j=1,\ldots,n$.
\end{lem}
\begin{proof}
The implication ``$\Rightarrow$'' is easy, so let us assume that the given determinant is identically vanishing on $A^n$.
Choose as many points $z_1,\ldots, z_r$ in $A$ as possible such that the rows
$$
(f_1(z_1), \dots, f_n(z_1))
$$
$$
 \vdots
$$
$$
(f_1(z_r), \dots, f_n(z_r))
$$
are linearly independent. By the assumption on the determinant $D$, we have $r < n$, hence there
exists a linear dependence between the columns, i.e., there are  {complex} numbers $a_1,\ldots,
a_n$,  {not all zero}, such that
\begin{equation}\label{eq:aifi}
 a_1 f_1(z_j) + \dots + a_n f_n(z_j) = 0
\end{equation}
for every $j \leq  r$.

Now we claim that this implies
\begin{equation}\label{eq:aiHi}
\sum a_i f_i = 0 \mbox{ on $A$},
\end{equation}
with $a_i$ as in (\ref{eq:aifi}).
To
verify this, choose any other point $z$ in $A$. By the choice of $z_1, \ldots, z_r$, the row
$$
(f_1(z), \dots, f_n(z))
$$
can be written as a linear combination of the rows
$$
(f_1(z_j), \dots, f_n(z_j)).
$$
This implies that (\ref{eq:aifi}) also holds for
$$
(f_1(z), \dots, f_n(z)),
$$
but this implies (\ref{eq:aiHi}).
\end{proof}

\begin{proof}[Proof of Proposition \ref{cor:transfer:indep:basic}]

(1) Consider the function $D$ in $\cCexp(X^\ell {\times } Y)$ given by
\[
D_{F,\psi}(x_1, \dots, x_\ell, y) = \det((H_{i,F,\psi}(x_j, y))_{ij})
.
\]
For each $F$, $\psi$ and $y$,
by Lemma \ref{lem:indep}, $D_{F,\psi}(\cdot, y)$ is identically zero on $X_F^\ell$ iff the $H_{i,F,\psi}(\cdot, y)$ for $i=1,\ldots,\ell$ are linearly dependent. Thus the statement we want to transfer is that $D_{F,\psi}$ is identically zero on $X_F^\ell \times Y_F$
for all $\psi$. This follows from \cite[Proposition 9.2.1]{CLexp} (which is the case of Theorem \ref{thm:transfer-gen} with $G=0$).

(2) Set $W := X^\ell$ and define  {$D$ in $\cCexp(W\times Y)$ as in (1).}

Consider $F$, $\psi$, $w = (x_1, \dots, x_\ell)$, $y$ such that $d := D_{F,\psi}(w, y) \ne 0$.
Then there exist  {unique} $c_1,\dots, c_\ell \in \CC$ such that
\begin{equation}\label{coeff:xj}
G_{F,\psi}(x_j, y) = \sum_i c_i H_{i,F,\psi}(x_j, y) \qquad\text{for } 1 \le j \le \ell
.
\end{equation}
By Cramer's rule, the products $c_i \cdot d$ are polynomials in $G_{F,\psi}(x_j, y)$ and $H_{i,F,\psi}(x_j, y)$,
so there exist functions $C_{i}$ in $\cCexp(W \times Y)$ such that
$c_i = C_{i,F,\psi}(w, y)/D_{F,\psi}(w, y)$. These $C_i$ (and this $D$) are as required:
As noted in the proof of (1), if $F$, $\psi$ and $y$ are such that the $H_{i,F,\psi}(\cdot, y)$ are linearly independent,
then there exists a $w \in W_F$ such that $D_{F,\psi}(w, y) \ne 0$, and if $G_{F,\psi}(\cdot, y)$ is a linear combintation of the
$H_{i,F,\psi}(\cdot, y)$, then for such a $w$, the coefficients $c_i$ from (\ref{coeff:xj}) are the desired ones.
\end{proof}

\bibliographystyle{amsplain}
\bibliography{tibib}
\end{document}